\documentclass{amsart}

\usepackage{amsthm}
\usepackage{amsmath}
\usepackage{mathtools}

\usepackage[pdfusetitle]{hyperref}

\newcommand\C{\mathbb{C}}
\newcommand\R{\mathbb{R}}

\newcommand\eps{\varepsilon}

\newtheorem{lemma}{Lemma}
\newtheorem{proposition}[lemma]{Proposition}

\newtheorem{theorem}[lemma]{Theorem}
\newtheorem{remark}[lemma]{Remark}
\newtheorem{example}[lemma]{Example}

\begin{document}

\title[Eigenvectors from eigenvalues]{Eigenvectors from Eigenvalues: a survey of a basic identity in linear algebra}

\author{Peter B.~Denton}
\address{Department of Physics, Brookhaven National Laboratory, Upton, NY 11973, USA}
\email{pdenton@bnl.gov}
\thanks{PBD acknowledges the United States Department of Energy under Grant Contract desc0012704 and the Fermilab Neutrino Physics Center.}

\author{Stephen J.~Parke}
\address{Theoretical Physics Department, Fermi National Accelerator Laboratory, Batavia, IL 60510, USA}
\email{parke@fnal.gov}
\thanks{This manuscript has been authored by Fermi Research Alliance, LLC under Contract No.~DE-AC02-07CH11359 with the U.S.~Department of Energy, Office of Science, Office of High Energy Physics.
FERMILAB-PUB-19-377-T}

\author{Terence Tao}
\address{Department of Mathematics, UCLA, Los Angeles CA 90095-1555}
\email{tao@math.ucla.edu}
 \thanks{TT was supported by a Simons Investigator grant, the James and Carol Collins Chair, the Mathematical Analysis \& Application Research Fund Endowment, and by NSF grant DMS-1764034}

\author{Xining Zhang}
\address{Enrico Fermi Institute \& Department of Physics, University of Chicago, Chicago, IL 60637, USA}
\email{xining@uchicago.edu}

\begin{abstract}  If $A$ is an $n \times n$ Hermitian matrix with eigenvalues $\lambda_1(A),\dots,\lambda_n(A)$ and $i,j = 1,\dots,n$, then the $j^{\mathrm{th}}$ component $v_{i,j}$ of a unit eigenvector $v_i$ associated to the eigenvalue $\lambda_i(A)$ is related to the eigenvalues $\lambda_1(M_j),\dots,\lambda_{n-1}(M_j)$ of the minor $M_j$ of $A$ formed by removing the $j^{\mathrm{th}}$ row and column by the formula
$$ |v_{i,j}|^2\prod_{k=1;k\neq i}^{n}\left(\lambda_i(A)-\lambda_k(A)\right)=\prod_{k=1}^{n-1}\left(\lambda_i(A)-\lambda_k(M_j)\right)\,.$$
We refer to this identity as the \emph{eigenvector-eigenvalue identity} and show how this identity can  also be used to extract the relative phases between the components of any given eigenvector.
Despite the simple nature of this identity and the extremely mature state of development of linear algebra, this identity was not widely known until very recently.  In this survey we describe the many times that this identity, or variants thereof, have been discovered and rediscovered in the literature (with the earliest precursor we know of appearing in 1834).  We also provide a number of proofs and generalizations of the identity.  
\end{abstract}

\date{\today}

\maketitle

\section{Introduction}

If $A$ is an $n \times n$ Hermitian matrix, we denote its $n$ real eigenvalues by $\lambda_1(A),\dots,\lambda_n(A)$.  The ordering of the eigenvalues will not be of importance in this survey, but for sake of concreteness let us adopt the convention of non-decreasing eigenvalues:
$$ \lambda_1(A) \leq \dots \leq \lambda_n(A).$$
If $1 \leq j \leq n$, let $M_j$ denote the $n-1 \times n-1$ minor formed from $A$ by deleting the $j^{\mathrm{th}}$ row and column from $A$.  This is again a Hermitian matrix, and thus has $n-1$ real eigenvalues $\lambda_1(M_j),\dots,\lambda_{n-1}(M_j)$, which for sake of concreteness we again arrange in non-decreasing order.  In particular we have the well known \emph{Cauchy interlacing inequalities} (see e.g., \cite[p. 103-104]{wilkinson})
\begin{equation}\label{interlace}
 \lambda_i(A) \leq \lambda_i(M_j) \leq \lambda_{i+1}(A)
\end{equation}
for $i=1,\dots,n-1$.

By the spectral theorem, we can find an orthonormal basis of eigenvectors $v_1,\dots,v_n$ where the $v_i$ are in $\mathbb C_n$ of $A$ associated to the eigenvalues $\lambda_1(A),\dots,\lambda_n(A)$ respectively. For any $i,j=1,\dots,n$, let $v_{i,j}$ denote the $j^{\mathrm{th}}$ component of $v_i$.  This survey paper is devoted to the following elegant relation, which we will call the \emph{eigenvector-eigenvalue identity}, relating this eigenvector component to the eigenvalues of $A$ and $M_j$:

\begin{theorem}[Eigenvector-eigenvalue identity]\label{eei}  With the notation as above, we have
\begin{equation}
|v_{i,j}|^2\prod_{k=1;k\neq i}^{n}\left(\lambda_i(A)-\lambda_k(A)\right)=\prod_{k=1}^{n-1}\left(\lambda_i(A)-\lambda_k(M_j)\right)\,.
\label{eq:wts}
\end{equation}
\end{theorem}

If one lets $p_A: \C \to \C$ denote the characteristic polynomial of $A$,
\begin{equation}\label{pa-def}
 p_A(\lambda) \coloneqq \mathrm{det}(\lambda I_n - A) = \prod_{k=1}^n (\lambda - \lambda_k(A)),
\end{equation}
where $I_n$ denotes the $n \times n$ identity matrix, and similarly let $p_{M_j}: \C \to \C$ denote the characteristic polynomial of $M_j$,
$$ p_{M_j}(\lambda) \coloneqq \mathrm{det}(\lambda I_{n-1} - M_j) = \prod_{k=1}^{n-1} (\lambda - \lambda_k(M_j))$$
then the derivative $p'_A(\lambda_i(A))$ of $p_A$ at $\lambda = \lambda_i(A)$ is equal to
$$ p'_A(\lambda_i(A)) = \prod_{k=1;k\neq i}^{n}\left(\lambda_i(A)-\lambda_k(A)\right)$$
and so \eqref{eq:wts} can be equivalently written in the characteristic polynomial form
\begin{equation}\label{eq:wts-alt}
|v_{i,j}|^2 p'_A(\lambda_i(A))  = p_{M_j}(\lambda_i(A)).
\end{equation}

\begin{example} If we set $n=3$ and
\begin{align*}
A &= \begin{pmatrix}
1 & 1 & -1 \\
1 & 3 & 1\\
-1 & 1 &  3
\end{pmatrix}
\end{align*}
then the eigenvectors and eigenvalues are
\begin{align*}
 v_1&= \frac{1}{\sqrt{6}}\begin{pmatrix}
2  \\ -1\\ 1
\end{pmatrix} ; \quad \lambda_1(A)=0 \\ 
 v_2&= \frac{1}{\sqrt{3}}  \begin{pmatrix}
1  \\ 1\\ -1
\end{pmatrix}; \quad \lambda_2(A)=3 \\
 v_3&= \frac{1}{\sqrt{2}}  \begin{pmatrix}
0  \\ 1\\ 1
\end{pmatrix}; \quad \lambda_3(A)=4 
\end{align*}
with minors $M_j$ and eigenvalues $\lambda_i(M_j)$ given by
\begin{align*}
M_1 &= \begin{pmatrix}
3  & 1 \\
1  & 3
\end{pmatrix}; \quad \quad \lambda_{1,2}(M_1)=2, ~4 \\
M_2 &= \begin{pmatrix}
1  & -1 \\
 -1 & 3
\end{pmatrix};  \quad \quad \lambda_{1,2}(M_2)=2\mp \sqrt{2} \approx 0.59, ~3.4  \\
M_3 &= \begin{pmatrix}
1  & 1 \\
 1 & 3
\end{pmatrix}; \quad \quad \lambda_{1,2}(M_3)=2\mp \sqrt{2}   \approx 0.59, ~3.4;
\end{align*}
one can observe the interlacing inequalities \eqref{interlace}.  One can then verify \eqref{eq:wts} for all $i,j=1,2,3$:
\begin{align*}
\frac{2}{3} = |v_{1,1}|^2 &= \frac{(0-2)(0-4)}{(0-3)(0-4)}\\
\frac{1}{6} = |v_{1,2}|^2 &= \frac{(0-2-\sqrt{2})(0-2+\sqrt{2})}{(0-3)(0-4)}\\
\frac{1}{6} = |v_{1,3}|^2 &= \frac{(0-2-\sqrt{2})(0-2+\sqrt{2})}{(0-3)(0-4)}\\[3mm]
\frac{1}{3} = |v_{2,1}|^2 &= \frac{(3-2)(3-4)}{(3-0)(3-4)} \\
\frac{1}{3} = |v_{2,2}|^2 &= \frac{(3-2-\sqrt{2})(3-2+\sqrt{2})}{(3-0)(3-4)} \\
\frac{1}{3} = |v_{2,3}|^2 &= \frac{(3-2-\sqrt{2})(3-2+\sqrt{2})}{(3-0)(3-4)} \\[3mm]
0 = |v_{3,1}|^2 &= \frac{(4-2)(4-4)}{(4-0)(4-3)} \\
\frac{1}{2} = |v_{3,2}|^2 &= \frac{(4-2-\sqrt{2})(4-2+\sqrt{2})}{(4-0)(4-3)}\\
\frac{1}{2} = |v_{3,3}|^2 &= \frac{(4-2-\sqrt{2})(4-2+\sqrt{2})}{(4-0)(4-3)}.
\end{align*}
One can also verify \eqref{eq:wts-alt} for this example after computing
\begin{align*}
p'_A(\lambda) &= 3\lambda^2-14\lambda+12 \\
p_{M_1}(\lambda) &= \lambda^2 - 6\lambda + 8 \\
p_{M_2}(\lambda) &= \lambda^2 - 4\lambda + 2 \\
p_{M_3}(\lambda) &= \lambda^2 - 4\lambda + 2.
\end{align*}
Note, that $p'_A(\lambda) =p_{M_1}(\lambda)+p_{M_2}(\lambda)+p_{M_3}(\lambda)$, which  is needed for the column normalization, see item \textup{(x)} in the consistency checks below.
\end{example}

Numerical code to verify the identity can be found at \cite{EEIcode}.

Theorem \ref{eei} passes a number of basic consistency checks:

\begin{itemize}
\item[(i)] (Dilation symmetry) If one multiplies the matrix $A$ by a real scalar $c$, then the eigenvalues of $A$ and $M_j$ also get multiplied by $c$, while the coefficients $v_{i,j}$ remain unchanged, which does not affect the truth of \eqref{eq:wts}.  To put it another way, if one assigns units to the entries of $A$, then the eigenvalues of $A,M_j$ acquire the same units, while $v_{i,j}$ remains dimensionless, and the identity \eqref{eq:wts} is dimensionally consistent.
\item[(ii)] (Translation symmetry) If one adds a scalar multiple of the identity $\lambda I_n$ to $A$, then the eigenvalues of $A$ and $M_j$ are shifted by $\lambda$, while the coefficient $v_{i,j}$ remains unchanged.  Thus both sides of \eqref{eq:wts} remain unaffected by such a transformation.
\item[(iii)]  (Permutation symmetry) Permuting the eigenvalues of $A$ or $M_j$ does not affect either side of \eqref{eq:wts} (provided one also permutes the index $i$ accordingly).  Permuting the ordering of the rows (and colums), as well as the index $j$, similarly has no effect on \eqref{eq:wts}.
\item[(iv)]  (First degenerate case) If $v_{i,j}$ vanishes, then the eigenvector $v_i$ for $A$ also becomes an eigenvector for $M_j$ with the same eigenvalue $\lambda_i(A)$ after deleting the $j^{\mathrm{th}}$ coefficient.  In this case, both sides of \eqref{eq:wts} vanish.  
\item[(v)]  (Second degenerate case) If the eigenvalue $\lambda_i(A)$ of $A$ occurs with multiplicity greater than one, then by the interlacing inequalities \eqref{interlace} it also occurs as an eigenvalue of $M_j$.  Again in this case, both sides of \eqref{eq:wts} vanish.
\item[(vi)]  (Compatibility with interlacing) More generally, the identity \eqref{eq:wts} is consistent with the interlacing \eqref{interlace} because the component $v_{i,j}$ of the unit eigenvector $v_i$ has magnitude at most $1$.
\item[(vii)]  (Phase symmetry) One has the freedom to multiply each eigenvector $v_i$ by an arbitrary complex phase $e^{\sqrt{-1}\theta_i}$ without affecting the matrix $A$ or its minors $M_j$.  But both sides of \eqref{eq:wts} remain unchanged when one does so.
\item[(viii)]  (Diagonal case) If $A$ is a diagonal matrix with diagonal entries $\lambda_1(A),\dots,\lambda_n(A)$, then $|v_{i,j}|$ equals $1$ when $i=j$ and zero otherwise, while the eigenvalues of $M_j$ are formed from those of $A$ by deleting one copy of $\lambda_i(A)$.  In this case one can easily verify \eqref{eq:wts} by hand.
\item[(ix)]  (Row normalization)  As the eigenvectors $v_1,\dots,v_n$ form the columns of an orthogonal matrix, one must have the identity $\sum_{i=1}^n |v_{i,j}|^2 = 1$ for all $j=1,\dots,n$.  Assuming for simplicity that the eigenvalues $\lambda_i(A)$ are distinct, this follows easily from the algebraic identity $\sum_{i=1}^n \frac{\lambda_i^m}{\prod_{k=1; k \neq i}^n (\lambda_i-\lambda_k)} = \delta_{m,n-1}$ for $m=0,\dots,n-1$ and any distinct complex numbers $\lambda_1,\dots,\lambda_n$, which can be seen by integrating the rational function $\frac{z^m}{\prod_{k=1}^n (z-\lambda_k)}$ along a large circle $\{|z|=R\}$ and applying the residue theorem. See also Remark \ref{reverse} below.
\item[(x)]  (Column normalization) As the eigenvectors $v_1,\dots,v_n$ are unit vectors, one must have $\sum_{j=1}^n |v_{i,j}|^2 = 1$ for all $i=1,\dots,n$.  To verify this, the translation symmetry (ii) to normalize $\lambda_i(A)=0$, and then observe (e.g., from \eqref{adja}) that $(-1)^n \sum_{j=1}^n p_{M_j}(0) = \sum_{j=1}^n \mathrm{det}(M_j) = \mathrm{tr}\ \mathrm{adj}(A)$ is the $(n-1)^{\mathrm{th}}$ elementary symmetric function of the eigenvalues and thus equal (since $\lambda_i(A)$ vanishes) to $\prod_{k=1;k \neq i}^n \lambda_k(A) = (-1)^n p'_A(0)$.  Comparing this with \eqref{eq:wts-alt} we obtain $\sum_{j=1}^n |v_{i,j}|^2=1$ as desired.  Alternatively, one can see from Jacobi's formula $\frac{d}{dt} \det(A(t)) = \mathrm{tr}( \mathrm{adj}(A(t)) \frac{dA(t)}{dt} )$ that $p'_A(\lambda) = \mathrm{tr}( \mathrm{adj}(\lambda I_n - A) ) =\sum_j p_{M_j}(\lambda)$ which when combined with \eqref{eq:wts-alt} also recovers the identity $\sum_{j=1}^n |v_{i,j}|^2=1$.  Jacobi's formula give us the need relationships between the eigenvalues of A and the eigenvalues of the $M_j$'s. They are $ (n-k) S_k(A) = \sum^{n}_{j=1} S_k(M_j) $ for $k=1,\dots,n-1$. where $S_k(A)$ is the k-th elementary symmetric polynomial of the eigenvalues of A, e.g.~$S_1(A)=\mathrm{tr}(A), \cdots, S_n(A)=\mathrm{det}(A)$.
\item[(xi)] (Relative phase information) As mentioned in (vii) above, the phase of any individual eigenvector $v_{i}$ is arbitrary, therefore the relative phase between  $v_{i,k}$ and $v_{j,k}$,  $i \neq j$, is arbitrary.  However, the relative phases between  the components of any $v_i$, say between $v_{i,j}$ and $v_{i,k}$ for $j \neq k$, is not arbitrary.  \eqref{eq:wts} can be used to extract these relative phases as follows: consider a unitary transformation on the matrix A and its eigenvectors such that $v_{i,j} \rightarrow \frac{1}{\sqrt{2}}(v_{i,j} + \omega \, v_{i,k})$ and $v_{i,k} \rightarrow \frac{1}{\sqrt{2}}(v_{i,k} -\omega^* \, v_{i,j})$ with $\omega =1~ \text{or}  ~\sqrt{-1}$ where $\omega^*$ is the complex conjugate of $\omega$. Applying \eqref{eq:wts} to the original A and to the two unitary transformed A's, gives us the information need to extract $\mathrm{arg}(v_{i,j}v^*_{i,k})$.  Note $p_{M_j}(\lambda)$ and  $p_{M_k}(\lambda)$ are not invariant under this particular unitary transformation, but $p^\prime_A(\lambda)$ and other $p_{M_l}(\lambda)$, $l \neq j ~\text{or} ~k$,  are invariant. Further more the  unitarity condition that  the $\sum_{i=1}^n  v_{i,j}v^*_{i,k} =0$ for $j \neq k$,  can also be derived in this fashion.  
 For further discussion on the relative phases see beginning of Section \ref{discuss-sec}.
\end{itemize}

We also note that, since \eqref{eq:wts} is a continuous function of the matrix $A$, it is possible to treat all eigenvalues as simple via the usual limiting argument.

The eigenvector-eigenvalue identity has a surprisingly complicated history in the literature, having appeared in some form or another (albeit often in a lightly disguised form) in over two dozen references, and being independently rediscovered a half-dozen times, in fields as diverse as numerical linear algebra, random matrix theory, inverse eigenvalue problems, graph theory (including chemical graph theory, graph reconstruction, and walks on graphs), and neutrino physics; see Figure \ref{fig:graph}.  While the identity applies to all Hermitian matrices, and extends in fact to normal matrices and more generally to diagonalizable matrices, it has found particular application in the special case of symmetric tridiagonal matrices (such as Jacobi matrices), which are of particular importance in several fundamental algorithms in numerical linear algebra.  

While the eigenvector-eigenvalue identity is moderately familiar to some mathematical communities, it is not as broadly well known as other standard identities in linear algebra such as Cramer's rule \cite{cramer} or the Cauchy determinant formula \cite{cauchy} (though, as we shall shortly see, it can be readily derived from either of these identities).  While several of the papers in which some form of the identity was discovered went on to be cited several times by subsequent work, the citation graph is only very weakly connected; in particular, Figure \ref{fig:graph} reveals that many of the citations coming from the earliest work on the identity did not propagate to later works, which instead were based on independent rediscoveries of the identity (or one of its variants).  In many cases, the identity was not highlighted as a relation between eigenvectors and eigenvalues, but was instead introduced in passing as a tool to establish some other application; also, the form of the identity and the notation used varied widely from appearance to appearance, making it difficult to search for occurrences of the identity by standard search engines.  The situation changed after a popular science article \cite{wolchover-2019} reporting on the most recent rediscovery \cite{Denton:2019ovn, DPTZ} of the identity by ourselves; in the wake of the publicity generated by that article, we received many notifications (see Section \ref{ack}) of the disparate places in the literature where the eigenvector-eigenvalue identity, or an identity closely related to it, was discovered.  Effectively, this crowdsourced the task of collating all these references together.  In this paper, we survey all the appearances of the eigenvector-eigenvalue identity that we are aware of as a consequence of these efforts, as well as provide several proofs, generalizations, and applications of the identity.  Finally, we speculate on some reasons for the limited nature of the dissemination of this identity in prior literature. 

\begin{figure} [t]
\centering
\includegraphics[width=4.5in]{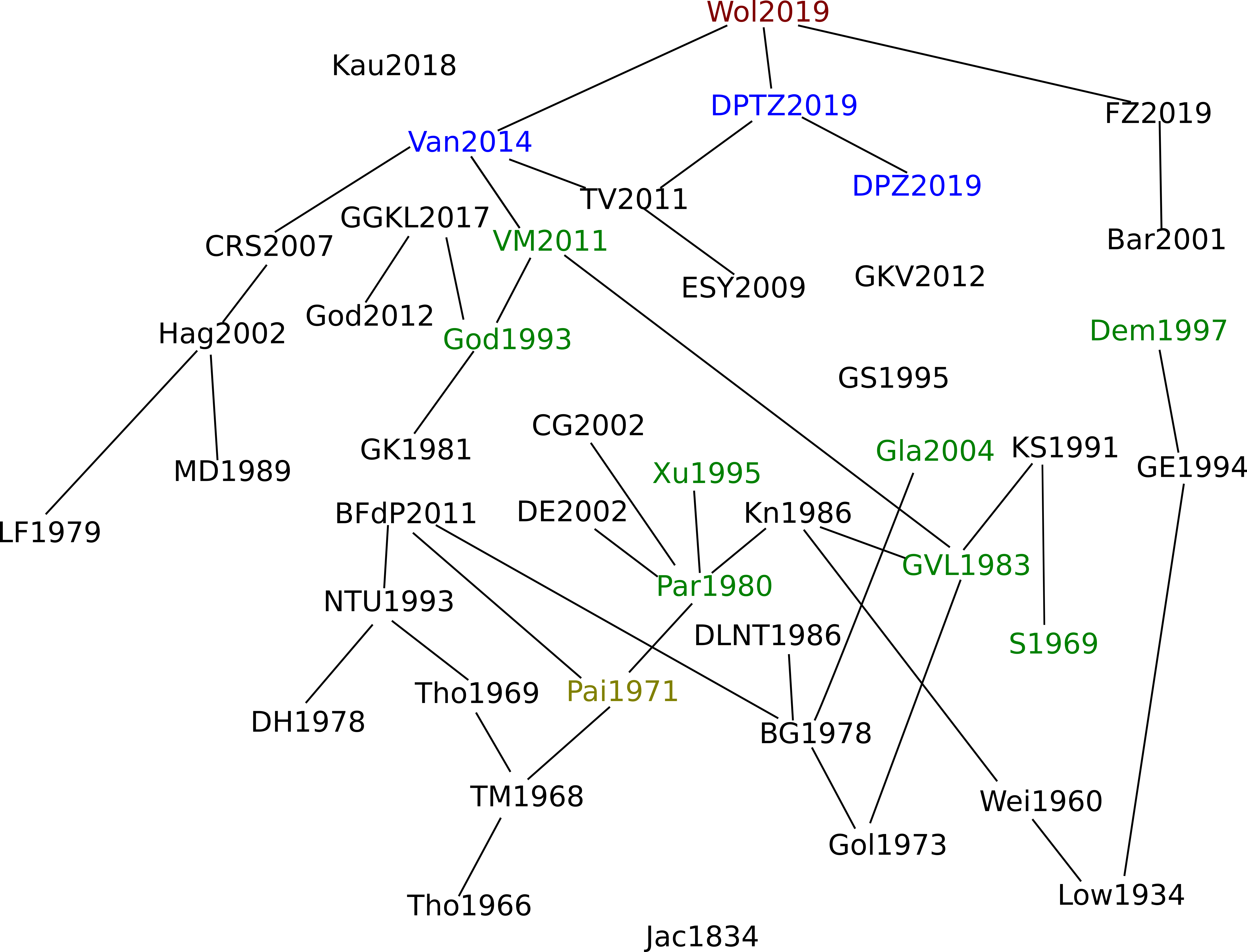}
\caption{The citation graph of all the references in the literature we are aware of (predating the current survey) that mention some variant of the eigenvector-eigenvalue identity.  To reduce clutter, transitive references (e.g., a citation of a paper already cited by another paper in the bibliography) are omitted.  Note the very weakly connected nature of the graph, with many early initial references not being (transitively) cited by many of the more recent references.  Blue references are preprints, green references are books, the brown reference is a thesis, and the red reference is a popular science article.  This graph was mostly crowdsourced from feedback received by the authors after the publication of \cite{wolchover-2019}.  The reference \cite{jacobi} predates all others found by a century!}
\label{fig:graph}
\end{figure}

\section{Proofs of the identity}

The identity \eqref{eq:wts} can be readily established from existing standard identities in the linear algebra literature.  We now give several such proofs.

\subsection{The adjugate proof}\label{adjugate-sec}

We first give a proof using adjugate matrices, which is a purely ``polynomial'' proof that avoids any invertibility, division, or non-degeneracy hypotheses in the argument; in particular, as we remark below, it has an extension to (diagonalizable) matrices that take values in arbitrary commutative rings.  This argument appears for instance in \cite[Section 7.9]{Parlett}.

Recall that if $A$ is an $n \times n$ matrix, the \emph{adjugate matrix} $\mathrm{adj}(A)$ is given by the formula
\begin{equation}\label{adj}
 \mathrm{adj}( A ) \coloneqq \left( (-1)^{i+j} \mathrm{det}(M_{ji}) \right)_{1 \leq i, j \leq n}
\end{equation}
where $M_{ji}$ is the $n-1 \times n-1$ matrix formed by deleting the $j^{\mathrm{th}}$ row and $i^{\mathrm{th}}$ column from $A$.  From Cramer's rule we have the identity
$$ \mathrm{adj}( A ) A = A \mathrm{adj}( A ) = \mathrm{det}(A) I_n.$$
If $A$ is a diagonal matrix with (complex) entries $\lambda_1(A),\dots,\lambda_n(A)$, then $\mathrm{adj}(A)$ is also a diagonal matrix with $ii$ entry $\prod_{k=1; k \neq i}^n \lambda_k(A)$.  More generally, if $A$ is a normal matrix with diagonalization
\begin{equation}\label{adiag}
 A = \sum_{i=1}^n \lambda_i(A) v_i v_i^*
\end{equation}
where $v_1,\dots,v_n$ are an orthonormal basis of eigenvectors of $A$ and $v_i^*$ is the conjugate transpose of $v_i$, then $\mathrm{adj}(A)$ has the same basis of eigenvectors with diagonalization
\begin{equation}\label{adja}
 \mathrm{adj}(A) = \sum_{i=1}^n (\prod_{k=1; k \neq i}^n \lambda_k(A)) v_i v_i^*.
\end{equation}
If one replaces $A$ by $\lambda I_n - A$ for any complex number $\lambda$, we therefore have\footnote{To our knowledge, this identity first appears in \cite[p. 157]{halmos}.}
$$ \mathrm{adj}(\lambda I_n - A) = \sum_{i=1}^n (\prod_{k=1; k \neq i}^n (\lambda - \lambda_k(A)) ) v_i v_i^*.$$
If one specializes to the case $\lambda = \lambda_i(A)$ for some $i=1,\dots,n$, then all but one of the summands on the right-hand side vanish, and the adjugate matrix becomes a scalar multiple of the rank one projection $v_i v_i^*$:
\begin{equation}\label{adji}
\mathrm{adj}(\lambda_i(A) I_n - A) = (\prod_{k=1; k \neq i}^n (\lambda_i(A) - \lambda_k(A)) ) v_i v_i^*.
\end{equation}
Extracting out the $jj$ component of this identity using \eqref{adj}, we conclude that
\begin{equation}\label{adji-2}
 \mathrm{det}(\lambda_i(A) I_{n-1} - M_j) = (\prod_{k=1; k \neq i}^n (\lambda_i(A) - \lambda_k(A)) ) |v_{i,j}|^2
\end{equation}
which is equivalent to \eqref{eq:wts}.  In fact this shows that the eigenvector-eigenvalue identity holds for normal matrices $A$ as well as Hermitian matrices (despite the fact that the minor $M_j$ need not be Hermitian or normal in this case).  Of course in this case the eigenvalues are not necessarily real and thus cannot be arranged in increasing order, but the order of the eigenvalues plays no role in the identity \eqref{eq:wts}.

\begin{remark}
The same argument also yields an off-diagonal variant
\begin{equation}\label{off-diag}
 (-1)^{j+j'} \mathrm{det}(\lambda_i(A) (I_n)_{j'j} - M_{j'j}) = (\prod_{k=1; k \neq i}^n (\lambda_i(A) - \lambda_k(A)) ) v_{i,j} \overline{v_{i,j'}}
\end{equation}
for any $1 \leq j,j' \leq n$, where $(I_n)_{j'j}$ is the $n-1 \times n-1$ minor of the identity matrix $I_n$.  When $j=j'$, this minor $(I_n)_{j'j}$ is simply equal to $I_{n-1}$ and the determinant can be expressed in terms of the eigenvalues of the minor $M_j$; however when $j \neq j'$ there is no obvious way to express the left-hand side of \eqref{off-diag} in terms of eigenvalues of $M_{j'j}$ (though one can still of course write the determinant as the product of the eigenvalues of $\lambda_i(A) (I_n)_{j'j} - M_{j'j}$).  Another way of viewing \eqref{off-diag} is that for every $1 \leq j' \leq n$, the vector with $j^{\mathrm{th}}$ entry
$$ (-1)^{j} \mathrm{det}(\lambda_i(A) (I_n)_{j'j} - M_{j'j}) $$
is a non-normalized eigenvector associated to the eigenvalue $\lambda_i(A)$; this observation appears for instance in \cite[p. 85--86]{gantmacher}.  See \cite{2014arXiv1401.4580V} for some further identities relating the components $v_{i,j}$ of the eigenvector $v_i$ to various determinants.
\end{remark}

\begin{remark} This remark is due to Vassilis Papanicolaou\footnote{\tt terrytao.wordpress.com/2019/08/13/eigenvectors-from-eigenvalues/\#comment-519905}.  The above argument also applies to non-normal matrices $A$, so long as they are diagonalizable with some eigenvalues $\lambda_1(A),\dots,\lambda_n(A)$ (not necessarily real or distinct).  Indeed, if one lets $v_1,\dots,v_n$ be a basis of right eigenvectors of $A$ (so that $A v_i = \lambda_i(A) v_i$ for all $i=1,\dots,n$), and let $w_1,\dots,w_n$ be the corresponding dual basis\footnote{In the case when $A$ is a normal matrix and the $v_i$ are unit eigenvectors, the dual eigenvector $w_i$ would be the complex conjugate of $v_i$.} of left eigenvectors (so $w_i^T A = w_i^T \lambda_i(A)$, and $w_i^T v_j$ is equal to $1$ when $i=j$ and $0$ otherwise) then one has the diagonalization
$$ A = \sum_{i=1}^n \lambda_i(A) v_i w_i^T$$
and one can generalize \eqref{adji} to
\begin{equation}\label{lia}
 \mathrm{adj}(\lambda_i(A) I_n - A) = (\prod_{k=1; k \neq i}^n (\lambda_i(A) - \lambda_k(A)) ) v_i w_i^T
\end{equation}
leading to an extension
\begin{equation}\label{adji-3}
 \mathrm{det}(\lambda_i(A) I_{n-1} - M_j) = (\prod_{k=1; k \neq i}^n (\lambda_i(A) - \lambda_k(A)) ) v_{i,j} w_{i,j}
\end{equation}
of \eqref{adji-2} to arbitrary diagonalizable matrices.  We remark that this argument shows that the identity \eqref{adji-3} is in fact valid for any diagonalizable matrix taking values in any commutative ring (not just the complex numbers).  The identity \eqref{off-diag} may be generalized in a similar fashion; we leave the details to the interested reader.  
\end{remark}

\begin{remark}\label{grin} As pointed out to us by Darij Grinberg\footnote{{\tt terrytao.wordpress.com/2019/12/03}}, the identity \eqref{eq:wts-alt} may be generalized to the non-diagonalizable setting.  Namely, one can prove that
$$ v_j w_j p'_A(\lambda) = (w^T v) p_{M_j}(\lambda)$$
for an arbitrary $n \times n$ matrix $A$ (with entries in an arbitrary commutative ring), any $j=1,\dots,n$ (with $M_j$ the $n-1 \times n-1$ minor formed from $A$ by removing the $j^{\mathrm{th}}$ row and column, and any right-eigenvector $Av = \lambda v$ and left-eigenvector $w^T A = w^T \lambda$ with a common eigenvalue $\lambda$.  After reducing to the case $\lambda=0$, the main step in the proof is to establishing two variants of \eqref{adji}, namely that $w_j (\mathrm{adj} A)_{i,k} = w_k (\mathrm{adj} A)_{i,j}$ and $(\mathrm{adj} A)_{k,i} v_j = (\mathrm{adj} A)_{j,i} v_k$ for all $i,j,k=1,\dots,n$. We refer the reader to the comment of Grinberg for further details.
\end{remark}

\begin{remark}  As observed in \cite[Appendix A]{2014arXiv1401.4580V}, one can obtain an equivalent identity to \eqref{eq:wts} by working with the powers $A^m, m=0,\dots,n-1$ in place of $\mathrm{adj}(A)$.  Indeed, from \eqref{adiag} we have
$$ A^m = \sum_{i=1}^n \lambda_i(A)^m v_i v_i^*$$
and hence on extracting the $jj$ component
$$ (A^m)_{jj} = \sum_{i=1}^n \lambda_i(A)^m |v_{i,j}|^2$$
for all $j=1,\dots,n$ and $m=0,\dots,n-1$.  Using Vandermonde determinants (and assuming for sake of argument that the eigenvalues $\lambda_i(A)$ are distinct), one can then solve for the $|v_{i,j}|^2$ in terms of the $(A^m)_{jj}$, eventually reaching an identity \cite[Theorem 2]{2014arXiv1401.4580V} equivalent to \eqref{eq:wts} (or \eqref{off-diag}), which in the case when $A$ is the adjacency matrix of a graph can also be expressed in terms of counts of walks of various lengths between pairs of vertices. We refer the reader to \cite{2014arXiv1401.4580V} for further details.
\end{remark}

\subsection{The Cramer's rule proof}\label{cramer-sec}

Returning now to the case of Hermitian matrices, we give a variant of the above proof of \eqref{eq:wts} that still relies primarily on Cramer's rule, but makes no explicit mention of the adjugate matrix; as discussed in the next section, variants of this argument have appeared multiple times in the literature.  We first observe that to prove \eqref{eq:wts} for Hermitian matrices $A$, it suffices to do so under the additional hypothesis that $A$ has simple spectrum (all eigenvalues occur with multiplicity one), or equivalently that
$$ \lambda_1(A) < \lambda_2(A) < \dots < \lambda_n(A).$$
This is because any Hermitian matrix with repeated eigenvalues can be approximated to arbitrary accuracy by a Hermitian matrix with simple spectrum, and both sides of \eqref{eq:wts} vary continuously with $A$ (at least if we avoid the case when $\lambda_i(A)$ occurs with multiplicity greater than one, which is easy to handle anyway by the second degenerate case (iv) noted in the introduction).

As before, we diagonalize $A$ in the form \eqref{adiag}.  For any complex parameter $\lambda$ not equal to one of the eigenvalues $\lambda_i(A)$, The resolvent $(\lambda I_n - A)^{-1}$ can then also be diagonalized as
\begin{equation}\label{res}
 (\lambda I_n - A)^{-1} = \sum_{i=1}^n \frac{v_i v_i^*}{\lambda - \lambda_i(A)}.
\end{equation}
Extracting out the $jj$ component of this matrix identity using Cramer's rule \cite{cramer}, we conclude that
$$ \frac{\mathrm{det}(\lambda I_{n-1} - M_j)}{\mathrm{det}(\lambda I_n - A)} = \sum_{i=1}^n \frac{|v_{i,j}|^2}{\lambda - \lambda_i(A)}$$
which we can express in terms of eigenvalues as
\begin{equation}\label{lkmj}
 \frac{\prod_{k=1}^{n-1} (\lambda - \lambda_k(M_j))}{\prod_{k=1}^n (\lambda - \lambda_k(A))} = \sum_{i=1}^n \frac{|v_{i,j}|^2}{\lambda - \lambda_i(A)}.
\end{equation}
Both sides of this identity are rational functions in $\lambda$, and have a pole at $\lambda = \lambda_i(A)$ for any given $i=1,\dots,n$.  Extracting the residue at this pole, we conclude that
$$ \frac{\prod_{k=1}^{n-1} (\lambda_i(A) - \lambda_k(M_j))}{\prod_{k=1: k \neq i}^n (\lambda_i(A) - \lambda_k(A))} = |v_{i,j}|^2$$
which rearranges to give \eqref{eq:wts}.

\begin{remark}  One can view the above derivation of \eqref{eq:wts} from \eqref{lkmj} as a special case of the partial fractions decomposition
$$ \frac{P(t)}{Q(t)} = \sum_{Q(\alpha)=0} \frac{P(\alpha)}{Q'(\alpha)} \frac{1}{t-\alpha}$$
whenever $Q$ is a polynomial with distinct roots $\alpha_1,\dots,\alpha_n$, and $P$ is a polynomial of degree less than that of $Q$.  Equivalently, this derivation can be viewed as a special case of the Lagrange interpolation formula (see e.g., \cite[\S 25.2]{AS})
$$ P(t) = \sum_{i=1}^n P(\alpha_i) \prod_{j=1; j \neq i}^n \frac{t - \alpha_j}{\alpha_i-\alpha_j} $$
whenever $\alpha_1,\dots,\alpha_n$ are distinct and $P$ is a polynomial of degree less than $n$.
\end{remark}

\begin{remark} A slight variant of this proof was observed by Aram Harrow\footnote{{\tt twitter.com/quantum\_aram/status/1195185551667847170}}, inspired by the inverse power method for approximately computing eigenvectors numerically.  We again assume simple spectrum. Using the translation invariance noted in consistency check (ii) of the introduction, we may assume without loss of generality that $\lambda_i(A)=0$.  Applying the resolvent identity \eqref{res} with $\lambda$ equal to a small non-zero quantity $\eps$, we conclude that
\begin{equation}\label{epsin}
 (\eps I_n - A)^{-1} = \frac{v_i v_i^*}{\eps} + O(1).
\end{equation}
On the other hand, by Cramer's rule, the $jj$ component of the left-hand side is 
$$ \frac{\mathrm{det}(\eps I_{n-1} - M_j)}{\mathrm{det}(\eps I_n - A)} = \frac{p_{M_j}(\eps)}{p_A(\eps)} = \frac{p_{M_j}(0) + O(\eps)}{\eps p'_A(\eps) + O(\eps^2)}.$$
Extracting out the top order terms in $\eps$, one obtains \eqref{eq:wts-alt} and hence \eqref{eq:wts}.  A variant of this argument also gives the more general identity
\begin{equation}\label{lam}
 \sum_{i: \lambda_i(A) = \lambda_*} |v_{i,j}|^2 = \lim_{\lambda \to \lambda_*} \frac{p_{M_j}(\lambda) (\lambda - \lambda_*)}{p_A(\lambda)}
\end{equation}
whenever $\lambda_*$ is an eigenvalue of $A$ of some multiplicity $m \geq 1$.  Note when $m=1$ we can recover \eqref{eq:wts-alt} thanks to L'H\^{o}pital's rule.  The right-hand side of \eqref{lam} can also be interpreted as the residue of the rational function $p_{M_j}/p_A$ at $\lambda_*$.
\end{remark}

An alternate approach way to arrive at \eqref{eq:wts} from \eqref{lkmj} is as follows.  Assume for the sake of this argument that the eigenvalues of $M_j$ are all distinct from the eigenvalues of $A$.  Then we can substitute $\lambda = \lambda_k(M_j)$ in \eqref{lkmj} and conclude that
\begin{equation}\label{la}
\sum_{i=1}^n \frac{|v_{i,j}|^2}{\lambda_k(M_j) - \lambda_i(A)} = 0
\end{equation}
for $k=1,\dots,n-1$.  Also, since the $v_i$ form an orthonormal basis, we have from expressing the unit vector $e_j$ in this basis that
\begin{equation}\label{vij}
 \sum_{i=1}^n |v_{i,j}|^2 = 1
\end{equation}
This is a system of $n$ linear equations in $n$ unknowns $|v_{i,j}|^2$.  For sake of notation let use permutation symmetry to set $i=n$.  From a further application of Cramer's rule, one can then write
$$ |v_{n,j}|^2 = \frac{\mathrm{det}(S')}{\mathrm{det}(S)}$$
where $S$ is the $n \times n$ matrix with $ki$ entry equal to $\frac{1}{\lambda_k(M_j) - \lambda_i(A)}$ when $k=1,\dots,n-1$, and equal to $1$ when $k=n$, and $S'$ is the minor of $S$ formed by removing the $n^{\mathrm{th}}$ row and column.  Using the well known Cauchy determinant identity \cite{cauchy}
\begin{equation}\label{cauchy-det}
\mathrm{det} \left( \frac{1}{x_i-y_j} \right)_{1 \leq i,j \leq n} = \frac{\prod_{1 \leq j < i \leq n} (x_i-x_j)(y_i-y_j)}{\prod_{i=1}^n \prod_{j=1}^n (x_i-y_j)}
\end{equation}
and inspecting the asymptotics as $x_n \to \infty$, we soon arrive at the identities
$$ \mathrm{det}(S) = \frac{ \prod_{1 \leq l < k \leq n-1} (\lambda_k(M_j) - \lambda_l(M_j)) \prod_{1 \leq l < k \leq n} (\lambda_k(A) - \lambda_l(A))}{\prod_{l=1}^{n-1} \prod_{k=1}^n (\lambda_l(M_j) - \lambda_k(A))}$$
and
$$ \mathrm{det}(S') = \frac{ \prod_{1 \leq l < k \leq n-1} (\lambda_k(M_j) - \lambda_l(M_j)) (\lambda_k(A) - \lambda_l(A))}{\prod_{l=1}^{n-1} \prod_{k=1}^{n-1} (\lambda_l(M_j) - \lambda_k(A))}$$
and the identity \eqref{eq:wts} then follows after a brief calculation.

\begin{remark}\label{reverse} The derivation of the eigenvector-eigenvalue identity \eqref{eq:wts} from \eqref{la}, as well as the obvious normalization \eqref{vij}, is reversible.  Indeed, the identity \eqref{eq:wts} implies that the rational functions on both sides of \eqref{lkmj} have the same residues at each of their (simple) poles, and these functions decay to zero at infinity, hence they must agree by Liouville's theorem.  Specializing \eqref{lkmj} to $\lambda = \lambda_k(M_j)$ then recovers \eqref{la}, while comparing the leading asymptotics of both sides of \eqref{lkmj} as $\lambda \to \infty$ recovers \eqref{vij} (note this also establishes the consistency check (ix) from the introduction).  As the identity \eqref{la} involves the same quantities $v_{i,j}, \lambda_k(M_j), \lambda_i(A)$ as \eqref{eq:wts}, one can thus view \eqref{la} as an equivalent formulation of the eigenvector-eigenvalue identity, at least in the case when all the eigenvalues of $A$ are distinct; the identity \eqref{lkmj} (viewing $\lambda$ as a free parameter) can also be interpreted in this fashion.
\end{remark}

\begin{remark} The above resolvent-based arguments have a good chance of being extended to certain classes of infinite matrices (e.g., Jacobi matrices), or other Hermitian operators, particularly if they have good spectral properties (e.g., they are trace class).  Certainly it is well known that spectral projections of an operator to a single eigenvalue $\lambda$ can often be viewed as residues of the resolvent at that eigenvalue, in the spirit of \eqref{epsin}, under various spectral hypotheses of the operator in question.  The main difficulty is to find a suitable extension of Cramer's rule to infinite-dimensional settings, which would presumably require some sort of regularized determinant such as the Fredholm determinant.  We will not explore this question further here, however, as pointed out to us by Carlos Tomei (personal communication), for reasonable Hermitian infinite matrices $A$ such as Jacobi matrices, one can formulate an identity similar to \eqref{lkmj} for the upper left coefficient $\langle e_1, (\lambda - A) e_1 \rangle$ of the resolvent for $\lambda$ in the upper half-plane, which can for instance be used (in conjunction with the Herglotz representation theorem) to recover the spectral theorem for such matrices.
\end{remark}

\subsection{Coordinate-free proof}\label{cfree-sec}

We now give a proof that largely avoids the use of coordinates or matrices, essentially due to Bo Berndtsson\footnote{\tt terrytao.wordpress.com/2019/08/13/eigenvectors-from-eigenvalues/\#comment-519914}.  For this proof we assume familiarity with exterior algebra (see e.g., \cite[Chapter XVI]{BM}).  The key identity is the following statement.

\begin{lemma}[Coordinate-free eigenvector-eigenvalue identity]\label{cfeei}  Let $T: \C^n \to \C^n$ be a self-adjoint linear map that annihilates a unit vector $v$.  For each unit vector $f \in \C^n$, let $\Delta_T(f)$ be the determinant of the quadratic form $w \mapsto (Tw, w)_{\C^n}$ restricted to the orthogonal complement $f^\perp \coloneqq \{ w \in \C^n: (f,w)_{\C^n} = 0 \}$, where $(,)_{\C^n}$ denotes the Hermitian inner product on $\C^n$.  Then one has
\begin{equation}\label{vf}
 |(v,f)_{\C^n}|^2 \Delta_T(v) = \Delta_T(f)
\end{equation}
for all unit vectors $f \in \C^n$.
\end{lemma}

\begin{proof} The determinant of a quadratic form $w \mapsto (Tw,w)_{\C^n}$ on a $k$-dimensional subspace $V$ of $\C^n$ can be expressed as $(T \alpha, \alpha)_{\bigwedge^k \C^n} / (\alpha, \alpha)_{\bigwedge^k \C^n}$ for any non-degenerate element $\alpha$ of the $k^{\mathrm{th}}$ exterior power $\bigwedge^k V \subset \bigwedge^k \C^n$ (equipped with the usual Hermitian inner product $(,)_{\bigwedge^k \C^n}$), where the operator $T$ is extended to $\bigwedge^k \C^n$ in the usual fashion.  If $f \in \C^n$ is a unit vector, then the Hodge dual $*f \in \bigwedge^{n-1} \C^n$ is a unit vector in $\bigwedge^{n-1} (f^\perp)$, so that we have the identity
\begin{equation}\label{delta-f}
 \Delta_T(f) = (T(*f), *f)_{\bigwedge^{n-1}\C^n}.
\end{equation}
To prove \eqref{vf}, it thus suffices to establish the more general identity
\begin{equation}\label{gen}
 (f,v)_{\C^n} \Delta_T(v) (v,g)_{\C^n} = (T(*f), (*g))_{\bigwedge^{n-1}\C^n}
\end{equation}
for all $f,g \in \C^n$.  If $f$ is orthogonal to $v$ then $*f$ can be expressed as a wedge product of $v$ with an element of $\bigwedge^{n-2}\C^n$, and hence $T(*f)$ vanishes, so that \eqref{gen} holds in this case.  If $g$ is orthogonal to $v$ then we again obtain \eqref{gen} thanks to the self-adjoint nature of $T$.  Finally, when $f=g=v$ the claim follows from \eqref{delta-f}.  Since the identity \eqref{gen} is sesquilinear in $f,g$, the claim follows.
\end{proof}

Now we can prove \eqref{eq:wts}.  Using translation symmetry we may normalize $\lambda_i(A)=0$.  We apply Lemma \ref{cfeei} to the self-adjoint map $T: w \mapsto Aw$, setting $v$ to be the null vector $v = v_i$ and $f$ to be the standard basis vector $e_j$.  Working in the orthonormal eigenvector basis $v_1,\dots,v_n$ we have $\Delta(v) = \prod_{k=1; k \neq i}^n \lambda_k(A)$; working in the standard basis $e_1,\dots,e_n$ we have $\Delta_T(f) = \mathrm{det}( M_j ) = \prod_{k=1}^{n-1} \lambda_k(M_j)$.  Finally we have $(v,f)_{\C^n} = v_{i,j}$.  The claim follows.

\begin{remark} In coordinates, the identity \eqref{delta-f} may be rewritten as $\Delta_T(f) = f^* \mathrm{adj}(A) f$.  Thus we see that Lemma \ref{cfeei} is basically \eqref{adji} in disguise.  It seems likely that the variant identity in Remark \ref{grin} can also be established in a similar coordinate-free fashion.
\end{remark}

\subsection{Proof using perturbative analysis}\label{perturbative-sec}

Now we give a proof using perturbation theory, which to our knowledge first appears in \cite{Mukherjee1989}.  By the usual limiting argument we may assume that $A$ has simple eigenvalues.  Let $\eps$ be a small parameter, and consider the rank one perturbation $A + \eps e_j e_j^*$ of $A$, where $e_1,\dots,e_n$ is the standard basis.  From \eqref{pa-def} and cofactor expansion, the characteristic polynomial $p_{A+\eps e_j e_j^*}(\lambda)$ of this perturbation may be expanded as
$$ p_{A+\eps e_j e_j^*}(\lambda) = p_A(\lambda) - \eps p_{M_j}(\lambda) + O(\eps^2).$$
On the other hand, from perturbation theory the eigenvalue $\lambda_i( A+\eps e_j e_j^* )$ may be expanded as
$$ \lambda_i( A+\eps e_j e_j^* ) = \lambda_i(A) + \eps |v_{i,j}|^2 + O(\eps^2).$$
If we then Taylor expand the identity
$$ p_{A + \eps e_j e_j^*}( \lambda_i( A+\eps e_j e_j^* ) ) = 0$$
and extract the terms that are linear in $\eps$, we conclude that
$$ \eps |v_{i,j}|^2 p'_A(\lambda_i(A)) - \eps p_{M_j}(\lambda_i(A)) = 0$$
which gives \eqref{eq:wts-alt} and hence \eqref{eq:wts}.

\subsection{Proof using a Cauchy-Binet type formula}\label{cauchy-binet-sec}

Now we give a proof based on a Cauchy-Binet type formula, which is also related to Lemma \ref{cfeei}.  This argument first appeared in \cite{DPTZ}.

\begin{lemma}[Cauchy-Binet type formula]\label{cbtf}  Let $A$ be an $n \times n$ Hermitian matrix with a zero eigenvalue $\lambda_i(A) = 0$.  Then for any $n \times n-1$ matrix $B$, one has
$$ \mathrm{det}( B^* A B ) = (-1)^{n-1} p'_A(0) \left| \mathrm{det} \begin{pmatrix} B & v_i \end{pmatrix} \right|^2 $$
where $\begin{pmatrix} B & v_i \end{pmatrix}$ denotes the $n \times n$ matrix with right column $v_i$ and all remaining columns given by $B$.
\end{lemma}

\begin{proof} We use a perturbative argument related to that in Section \ref{perturbative-sec}.  Since $A v_i = 0$, $v_i^* A = 0$, and $v_i^* v_i = 1$, we easily confirm the identity
$$ \begin{pmatrix} B^* \\ v_i^* \end{pmatrix} (\eps I_n-A) \begin{pmatrix} B & v_i \end{pmatrix} = \begin{pmatrix} -B^* A B +O(\eps)& O(\eps) \\ O(\eps) & \eps \end{pmatrix}$$
for any parameter $\eps$, where the matrix on the right-hand side is given in block form, with the top left block being an $n-1 \times n-1$ matrix and the bottom right entry being a scalar.  Taking determinants of both sides, we conclude that
$$ p_A(\eps) \left| \mathrm{det} \begin{pmatrix} B & v_i \end{pmatrix} \right|^2 = (-1)^{n-1} \mathrm{det}( B^* A B^* ) \eps + O(\eps^2).$$
Extracting out the $\eps$ coefficient of both sides, we obtain the claim.
\end{proof}

\begin{remark}  In the case when $v_i$ is the basis vector $e_n$, we may write $A$ in block form as $A = \begin{pmatrix} M_n & 0_{n-1 \times 1} \\ 0_{1 \times n-1} & 0 \end{pmatrix}$, where $0_{i \times j}$ denotes the $i \times j$ zero matrix, and write $B = \begin{pmatrix} B' \\ x^* \end{pmatrix}$ for some $n-1 \times n-1$ matrix $B'$ and $n-1$-dimensional vector $x$, in which case one can calculate
$$ \mathrm{det}(B^* A B) = \mathrm{det}( (B')^* M_n B' ) = \mathrm{det}(M_n) |\mathrm{det}(B')|^2$$
and
$$ \mathrm{det} \begin{pmatrix} B & v_i \end{pmatrix} = \mathrm{det}(B').$$
Since $p'_A(0) = (-1)^{n-1} \mathrm{det}(M_n)$ in this case, this establishes \eqref{cbtf} in the case $v_i=e_n$.  The general case can then be established from this by replacing $A$ by $UAU^*$ and $B$ by $UB$, where $U$ is any unitary matrix that maps $v_i$ to $e_n$.
\end{remark}

We now prove \eqref{eq:wts-alt} and hence \eqref{eq:wts}.  Using the permutation and translation symmetries we may normalize $\lambda_i(A)=0$ and $j=1$.  If we then apply Lemma \ref{cbtf} with $B = \begin{pmatrix} 0_{1 \times n-1} \\ I_{n-1} \end{pmatrix}$, in which case
$$ \mathrm{det}( B^* A B ) = \mathrm{det}(M_1) = (-1)^{n-1} p_{M_1}(0)$$
and
$$ \mathrm{det} \begin{pmatrix} B & v_i \end{pmatrix} = v_{i,1}.$$
Applying Lemma \ref{cbtf}, we obtain \eqref{eq:wts-alt}.

\subsection{Proof using an alternate expression for eigenvector component magnitudes}

There is an alternate formula for the square $|v_{i,j}|^2$ of the eigenvector component $v_{i,j}$ that was introduced in the paper \cite[(5.8)]{ESY} of Erd\H{o}s, Schlein and Yau in the context of random matrix theory, and then highlighted further in a paper \cite[Lemma 41]{tao2011} of the third author and Vu; it differs from the eigenvector-eigenvalue formula in that it involves the actual coefficients of $A$ and $M_1$, rather than just their eigenvalues.  It was also previously discovered by Gaveau and Schulman \cite[(2.6)]{gaveau}. For sake of notation we just give the formula in the $j=1$ case.

\begin{lemma}[Alternate expression for $v_{i,1}$]\label{vn}  Let $A$ be an $n \times n$ Hermitian matrix written in block matrix form as
$$ A = \begin{pmatrix} a_{11} & X^* \\ X & M_1 \end{pmatrix}$$
where $X$ is an $n-1$-dimensional column vector and $a_{11}$ is a scalar.  Let $i=1,\dots,n$, and suppose that $\lambda_i(A)$ is not an eigenvalue of $M_1$.  Let $u_1,\dots,u_{n-1}$ denote an orthonormal basis of eigenvectors of $M_1$, associated to the eigenvalues $\lambda_1(M_1),\dots,\lambda_{n-1}(M_1)$.  Then
\begin{equation}\label{wts-3}
 |v_{i,1}|^2 = \frac{1}{1 + \sum_{j=1}^{n-1} X^* (\lambda_i(A) I_{n-1} - M_1)^{-2} X}.
\end{equation}
\end{lemma}

This lemma is useful in random matrix theory for proving \emph{delocalization} of eigenvectors of random matrices, which roughly speaking amounts to proving upper bounds on the quantity $\sup_{1 \leq j \leq n} |v_{i,j}|$.

\begin{proof} One can verify that this result enjoys the same translation symmetry as Theorem \ref{eei} (see consistency check (ii) from the introduction), so without loss of generality we may normalize $\lambda_i(A) = 0$.  If we write $v_i = \begin{pmatrix} v_{i,1} \\ w_i \end{pmatrix}$ for an $n-1$-dimensional column vector $w_i$, then the eigenvector equation $A v_i = \lambda_i(A) v_i = 0$ can be written as the system
\begin{align*}
a_{11} v_{i,1} + X^* w_i &= 0 \\
X v_{i,1} + M_1 w_i &= 0.
\end{align*}
By hypothesis, $0$ is not an eigenvalue of $M_1$, so we may invert $M_1$ and conclude that
$$ w_i = - M_1^{-1} X v_{i,1}.$$
Since $v_i$ is a unit vector, we have $|w_i|^2 + |v_{i,1}|^2 = 1$.  Combining these two formulae and using some algebra, we obtain the claim.
\end{proof}

Now we can give an alternate proof of \eqref{eq:wts-alt} and hence \eqref{eq:wts}.  By permutation symmetry (iii) it suffices to establish the $j=1$ case.  Using limiting arguments as before we may assume that $A$ has distinct eigenvalues; by further limiting arguments we may also assume that the eigenvalues of $M_1$ are distinct from those of $A$.  By translation symmetry (ii) we may normalize $\lambda_i(A) = 0$.  Comparing \eqref{eq:wts-alt} with \eqref{wts-3}, our task reduces to establishing the identity
$$ p'_A(0)  = p_{M_1}(0) (1 + X^* M_1^{-2} X).$$
However, for any complex number $\lambda$ not equal to an eigenvalue of $M_1$, we may apply Schur complementation \cite{cottle} to the matrix
$$ \lambda I_n - A = \begin{pmatrix} \lambda - a_{11} & -X^* \\ -X & \lambda I_{n-1} - M_1 \end{pmatrix} $$
to obtain the formula
$$ \mathrm{det}( \lambda I_n - A ) = \mathrm{det}( \lambda I_{n-1} - M_1 ) (\lambda - a_{11} - X^* (\lambda I_{n-1} - M_1)^{-1} X )$$
or equivalently
$$ p_A(\lambda) = p_{M_1}(\lambda) (\lambda - a_{11} - X^* (\lambda I_{n-1} - M_1)^{-1} X )$$
which on Taylor expansion around $\lambda=0$ using $p_A(0)=0$ gives
$$ p'_A(0) \lambda + O(\lambda^2) = (p_{M_1}(0) + O(\lambda)) ( \lambda - a_{11} + X^* M_1^{-1} X + \lambda X^* M_1^{-2} X + O(\lambda^2)).$$
Setting $\lambda=0$ and using $p_{M_1}(0) \neq 0$, we conclude that $a_{11} + X^* M_1^{-1} X$ vanishes.  If we then extract the $\lambda$ coefficient, we obtain the claim.

\begin{remark}  The same calculations also give the well known fact that the minor eigenvalues $\lambda_1(M_1),\dots,\lambda_{n-1}(M_1)$ are precisely the roots for the equation
$$ \lambda - a_{11} - X^* (\lambda I_{n-1} - M_1)^{-1} X = 0.$$
Among other things, this can be used to establish the interlacing inequalities \eqref{interlace}.
\end{remark}

\subsection{A generalization}

The following generalization of the eigenvector-eigenvalue identity was recently observed\footnote{Variants of this identity have also been recently observed in \cite{chen}, \cite{stawiska}.} by Yu Qing Tang (private communication), relying primarily on the Cauchy-Binet formula and a duality relationship \eqref{duality} between the various minors of a unitary matrix.  If $A$ is an $n \times n$ matrices and $I,J$ are subsets of $\{1,\dots,n\}$ of the same cardinality $m$, let $M_{I,J}(A)$ denote the $n-m \times n-m$ minor formed by removing the $m$ rows indexed by $I$ and the $m$ columns indexed by $J$.

\begin{proposition}[Generalized eigenvector-eigenvalue identity]  Let $A$ be a normal $n \times n$ matrix diagonalized as $A = UDU^*$ for some unitary $U$ and diagonal $D = \mathrm{diag}(\lambda_1,\dots,\lambda_n)$, let $1 \leq m < n$, and let $I,J,K \subset \{1,\dots,n\}$ have cardinality $m$.  Then
$$ \overline{\mathrm{det} M_{J^c,I^c}(U)} (\mathrm{det} M_{K^c,I^c}(U)) \prod_{i \in I, j \in I^c} (\lambda_j - \lambda_i) = \mathrm{det} M_{J,K}( \prod_{i \in I} (A - \lambda_i I_n) )$$
where $I^c$ denotes the complement of $I$ in $\{1,\dots,n\}$, and similarly for $J^c,K^c$.
\end{proposition}

Note that if we set $m=1$, $I = \{i\}$, and $J=K=\{j\}$, we recover \eqref{eq:wts}.  The identity \eqref{off-diag} can be interpreted as the remaining $m=1$ cases of this proposition.

\begin{proof}  We have 
$$\prod_{i \in I} (A - \lambda_i I_n) = U \prod_{i \in I} (D - \lambda_i I_n) U^*$$
and hence by the Cauchy-Binet formula 
$$\mathrm{det} M_{J,K}( \prod_{i \in I} (A - \lambda_i I_n) ) = \sum_{L,L'} (\mathrm{det} M_{J,L}(U))
(\mathrm{det} M_{L,L'}(\prod_{i \in I} (D - \lambda_i I_n))) (\mathrm{det} M_{L',K}(U^*))$$
where $L,L'$ range over subsets of $\{1,\dots,n\}$ of cardinality $m$.  A computation reveals that the quantity $\mathrm{det} M_{L,L'}(\prod_{i \in I} (D - \lambda_i I_n))$ vanishes unless $L=L'=I$, in which case the quantity equals $\prod_{i \in I, j \in I^c} (\lambda_j - \lambda_i)$.  Thus it remains to show that
$$ \overline{\mathrm{det} M_{J^c,I^c}(U)} (\mathrm{det} M_{K^c,I^c}(U)) = \mathrm{det} M_{J,I}(U) \mathrm{det} M_{I,K}(U^*).$$
Since $\mathrm{det} M_{I,K}(U^*) = \overline{\mathrm{det} M_{K,I}(U)}$, it will suffice to show that\footnote{This identity is also a special case of the more general identity $\mathrm{det} M_{I,J}(\mathrm{adj}(A)) = (\mathrm{det} A)^{m-1} \mathrm{det} M_{J^c,I^c}(A)$ of Jacobi \cite{jacobi}, valid for arbitrary $n \times n$ matrices $A$, as can be seen after noting that $\mathrm{adj}(U) = \mathrm{det}(U)^{-1} U^*$.} 
\begin{equation}\label{duality}
 \mathrm{det} M_{J,I}(U) = \overline{\mathrm{det} M_{J^c,I^c}(U)} \mathrm{det} U
\end{equation}
for any $J,I \subset \{1,\dots,n\}$ of cardinality $m$.  By permuting rows and columns we may assume that $J=I=\{1,\dots,m\}$.  If we split the identity matrix $I_n$ into the left $m$ columns $I_n^1 \coloneqq \begin{pmatrix} I_m \\ 0_{n-m \times m} \end{pmatrix}$ and the right $n-m$ columns $I_n^2 \coloneqq \begin{pmatrix} 0_{m \times n-m} \\ I_{n-m} \end{pmatrix}$ and take determinants of both sides of the identity
$$ U \begin{pmatrix} U^* I_n^1 & I_n^2 \end{pmatrix} = \begin{pmatrix} I_n^1 & UI_n^2 \end{pmatrix}$$
we conclude that
$$ \mathrm{det}(U) \mathrm{det} M_{I^c,J^c}(U^*) = \mathrm{det} M_{J,I}(U)$$
giving the claim.
\end{proof}

\section{History of the identity}\label{history-sec}

In this section we present, roughly in chronological order, all the references to the eigenvector-eigenvalue identity \eqref{eq:wts} (or closely related results) that we are aware of in the literature.  For the primary references, we shall present the identity in the notation of that reference in order to highlight the diversity of contexts and notational conventions in which this identity has appeared.
 
The earliest appearance of identities equivalent to \eqref{eq:wts} that we know of is due to Jacobi \cite[\S 8, (33)]{jacobi}.  In modern notation, Jacobi diagonalizes a symmetric quadratic form $\sum_{\chi=1}^n \sum_{\lambda=1}^n a_{\chi,\lambda} x_\chi x_\lambda$ as $\sum_{m=1}^n G_m (\sum_{i=1}^n \alpha_i^{(m)} x_i)^2$ for an orthogonal matrix $(\alpha_i^{(m)})_{1 \leq i,m \leq n}$, and then the for each $m$ the cofactors $B^{(m)}_{\chi \lambda}$ of the form $\sum_{\chi=1}^n \sum_{\lambda=1}^n a_{\chi,\lambda} x_\chi x_\lambda - G_m \sum_{\chi=1}^n x_\chi^2$ are extracted.  Noting that the columns of this cofactor matrix are proportional to the eigenvector $(\alpha_i^{(m)})_{1 \leq i \leq n}$, Jacobi concluded that
\begin{equation}\label{denom}
(G_1-G_m) \dots (G_n-G_m) \alpha_\chi^{(m)} \alpha_\lambda^{(m)} = B_{\chi\lambda}^{(m)}
\end{equation}
with the factor $G_m-G_m$ omitted from the left-hand side; this is essentially \eqref{adji} for real symmetric matrices. In \cite[\S 8, (36)]{jacobi} an identity essentially equivalent to \eqref{eq:wts-alt} for real symmetric matrices is also given.
 
An identity that implies \eqref{eq:wts} as a limiting case appears a century after Jacobi in a paper of L\"owner \cite[(7)]{Lowner}.  In this paper, a diagonal quadratic form
$$ A_n(x,x) = \sum_{i=1}^n \lambda_i x_i^2$$
is considered, as well as a rank one perturbation
$$ B_n(x,x) = A_n(x,x) + \left( \sum_{i=1}^n \alpha_i x_i \right)^2$$
for some real numbers $\lambda_1,\dots,\lambda_n,\alpha_1,\dots,\alpha_n$.  If the eigenvalues of the quadratic form $B_n$ are denoted $\mu_1,\dots,\mu_n$. If the eigenvalues are arranged in non-decreasing order, one has the interlacing inequalities
$$\lambda_1 \leq \mu_1 \leq \lambda_2 \leq \dots \leq \lambda_n \leq \mu_n$$
(compare with \eqref{interlace}).  Under the non-degeneracy hypothesis
$$\lambda_1 < \mu_1 < \lambda_2 < \dots < \lambda_n < \mu_n$$
the identity
\begin{equation}\label{alpha-k}
 \alpha_k^2 = \frac{(\mu_1-\lambda_k) (\mu_2 - \lambda_k) \dots (\mu_n - \lambda_k)}{(\lambda_1-\lambda_k) (\lambda_2 - \lambda_k) \dots (\lambda_{k-1} - \lambda_k) (\lambda_{k+1} - \lambda_k) \dots (\lambda_n - \lambda_k)}
\end{equation}
is established for $k=1,\dots,n$, which closely resembles \eqref{eq:wts}; a similar formula also appears in \cite[\S 12]{jacobi}.  The identity \eqref{alpha-k} is obtained via ``Eine einfache Rechnung'' (an easy calculation) from the standard relations
$$ \sum_{k=1}^n \frac{\alpha_k^2}{\mu_i - \lambda_k} = 1$$
for $i=1,\dots,n$ (compare with \eqref{la}), after applying Cramer's rule and the Cauchy determinant identity \eqref{cauchy-det}; as such, it is very similar to the proof of \eqref{eq:wts} in Section \ref{cramer-sec} that is also based on \eqref{cauchy-det}.   The identity \eqref{alpha-k} was used in \cite{Lowner} to help classify monotone functions of matrices, and has also been applied to inverse eigenvector problems and stable computation of eigenvectors \cite{GE}, \cite[p. 224-226]{demmel}.  It can be related to \eqref{eq:wts} as follows.  For sake of notation let us just consider the $j=n$ case of \eqref{eq:wts}.  Let $\eps$ be a small parameter and consider the perturbation $e_n e_n^* + \eps A$ of the rank one matrix $e_n e_n^*$.  Standard perturbative analysis reveals that the eigenvalues of this perturbation consist of $n-1$ eigenvalues of the form $\eps \lambda_i(M_n) + O(\eps^2)$ for $i=1,\dots,n-1$, plus an outlier eigenvalue at $1+O(\eps)$.  Rescaling, we see that the rank one perturbation $A + \frac{1}{\eps} e_n e_n^*$ of $A$ has eigenvalues of the form $\lambda_i(M_n) + O(\eps)$ for $i=1,\dots,n-1$, plus an outlier eigenvalue at $\frac{1}{\eps}+O(1)$.  If we let $A_n, B_n$ be the quadratic forms associated to $A, A + \frac{1}{\eps} e_n e_n^*$ expressed using the eigenvector basis $v_1,\dots,v_n$, the identity \eqref{alpha-k} becomes
$$ \frac{1}{\eps} |v_{k,n}|^2 = \frac{\prod_{i=1}^n (\lambda_i( A + \frac{1}{\eps} e_n e_n^* ) - \lambda_k(A))}{\prod_{i=1; i \neq k}^n (\lambda_i(A) - \lambda_k(A))}.$$
Extracting the $1/\eps$ component of both sides of this identity using the aforementioned perturbative analysis, we recover \eqref{eq:wts} after a brief calculation.

A more complicated variant of \eqref{alpha-k} involving various quantities related to the Rayleigh-Ritz method of bounding the eigenvalues of a symmetric linear operator was stated by Weinberger \cite[(2.29)]{weinberger}, where it was noted that it can be proven much the same method as in \cite{Lowner}.  

The first appearance of the eigenvector-eigenvalue identity in essentially the form presented here that we are aware of was by Thompson \cite[(15)]{Thompson:1966}, which does not reference the prior work of L\"owner or Weinberger.  In the notation of Thompson's paper, $A$ is a normal $n \times n$ matrix, and $\mu_1,\dots,\mu_s$ are the \emph{distinct} eigenvalues of $A$, with each $\mu_i$ occurring with multiplicity $e_i$.  To avoid non-degeneracy it is assumed that $s \geq 2$.	One then diagonalizes $A = UDU^{-1}$ for a unitary $U$ and diagonal $D = \mathrm{diag}(\lambda_1,\dots,\lambda_n)$, and then sets
$$ \theta_{i\beta} = \sum_{j: \lambda_j = \mu_\beta} |U_{ij}|^2$$
for $i=1,\dots,n$ and $\beta=1,\dots,n-1$, where $U_{ij}$ are the coefficients of $U$. The minor formed by removing the $i^{\mathrm{th}}$ row and column from $A$ is denoted $A(i|i)$; it has ``trivial'' eigenvalues in which each $\mu_i$ with $e_i>1$ occurs with multiplicity $e_i-1$, as well as additionally some ``non-trivial'' eigenvalues $\xi_{i1},\dots,\xi_{i,s-1}$.  The equation \cite[(15)]{Thompson:1966} then reads
\begin{equation}\label{t66}
 \theta_{i\alpha} = \prod_{j=1}^{s-1} (\mu_\alpha - \xi_{ij}) \prod_{j=1,j \neq \alpha}^s (\mu_\alpha - \mu_j)^{-1}
\end{equation}
for $1 \leq \alpha \leq s$ and $1 \leq i \leq n$.  If one specializes to the case when $A$ is Hermitian with simple spectrum, so that all the multiplicities $e_i$ are equal to $1$, and set $s=n$ and $\mu_i = \lambda_i$, it is then not difficult to verify that this identity is equivalent to the eigenvector-eigenvalue identity \eqref{eq:wts} in this simple spectrum case. In the case of repeated eigenvalues, the eigenvector-eigenvalue identity \eqref{eq:wts} may degenerate (in that the left and right hand sides both vanish), but the identity \eqref{t66} remains non-trivial in this case.  The proof of \eqref{t66} given in \cite{Thompson:1966} is written using a rather complicated notation (in part because much of the paper was concerned with more general $k \times k$ minors rather than the $n-1 \times n-1$ minors $A(i|i)$), but is essentially the adjugate proof from Section \ref{adjugate-sec} (where the adjugate matrix is replaced by the closely related $(n-1)^{\mathrm{th}}$ compound matrix).  In \cite{Thompson:1966}, the identity \eqref{t66} was not highlighted as a result of primary interest in its own right, but was instead employed to establish a large number of inequalities between the eigenvalues $\mu_1,\dots,\mu_s$ and the minor eigenvalues $\xi_{i1},\dots,\xi_{i,s-1}$ in the Hermitian case; see \cite[Section 5]{Thompson:1966}.

In followup paper \cite{Thompson:1968} by Thompson and McEnteggert, the analysis from \cite{Thompson:1966} was revisited, restricting attention specifically to the case of an $n \times n$ Hermitian matrix $H$ with simple eigenvalues $\lambda_1 < \dots < \lambda_n$, and with the minor $H(i|i)$ formed by deleting the $i^{\mathrm{th}}$ row and column having eigenvalues $\xi_{i1} \leq \dots \leq \xi_{i,n-1}$.  In this paper the inequalities
\begin{equation}\label{ineq1}
 \sum_{i=1}^n \frac{\lambda_j - \xi_{i,j-1}}{\lambda_j - \lambda_{j-1}} \frac{\xi_{ij} - \lambda_j}{\lambda_{j+1}-\lambda_j} \geq 1
\end{equation}
and
\begin{equation}\label{ineq2}
 \sum_{i=1}^n \frac{\lambda_j - \xi_{i,j-1}}{\lambda_j - \lambda_1} \frac{\xi_{ij}-\lambda_j}{\lambda_n-\lambda_j} \leq 1
\end{equation}
for $1 \leq j \leq n$ were proved (with most cases of these inequalities already established in \cite{Thompson:1966}), with a key input being the identity
\begin{equation}\label{uij}
|u_{ij}|^2 = \left\{ \frac{\lambda_j - \xi_{i1}}{\lambda_j - \lambda_1} \right\} \dots \left\{ \frac{\lambda_j - \xi_{i,j-1}}{\lambda_j - \lambda_{j-1}} \right\} \left\{ \frac{\xi_{ij}-\lambda_j}{\lambda_{j+1} - \lambda_j} \right\} \dots \left\{ \frac{\xi_{i,n-1}-\lambda_j}{\lambda_{n} - \lambda_j} \right\}
\end{equation}
where $u_{ij}$ are the components of the unitary matrix $U$ used in the diagonalization $H = UDU^{-1}$ of $H$.  Note from the Cauchy interlacing inequalities that each of the expressions in braces takes values between $0$ and $1$. It is not difficult to see that this identity is equivalent to \eqref{eq:wts} (or \eqref{t66}) in the case of Hermitian matrices with simple eigenvalues, and the hypothesis of simple eigenvalues can then be removed by the usual limiting argument.  As in \cite{Thompson:1966}, the identity is established using adjugate matrices, essentially by the argument given in the previous section. However, the identity \eqref{uij} is only derived as an intermediate step towards establishing the inequalities \eqref{ineq1}, \eqref{ineq2}, and is not highlighted as of interest in its own right.  The identity \eqref{t66} was then reproduced in a further followup paper \cite{thompson-iv}, in which the identity \eqref{lkmj} was also noted; this latter identity was also independently observed in \cite{deutsch}.

In the text of \v{S}ilov \cite[Section 10.27]{silov}, the identity \eqref{eq:wts} is established, essentially by the Cramer rule method.  Namely, if $A(x,x)$ is a diagonal real quadratic form on $\R^n$ with eigenvalues $\lambda_1 \geq \dots \geq \lambda_n$, and $R_{n-1}$ is a hyperplane in $\R^n$ with unit normal vector $(\alpha_1,\dots,\alpha_n)$, and $\mu_1 \geq \dots \geq \mu_{n-1}$ are the eigenvalues of $A$ on $R_{n-1}$, then it is observed that 
\begin{equation}\label{alphak2}
\alpha_k^2 = \frac{\prod_{k=1}^{n-1} (\mu_k - \lambda)}{(\lambda_k - \lambda_1) \dots (\lambda_k - \lambda_{k-1}) (\lambda_k - \lambda_{k+1}) \dots (\lambda_k-\lambda_n)}
\end{equation}
for $k=1,\dots,n$, which is \eqref{eq:wts} after changing to the eigenvector basis; identities equivalent to \eqref{la} and \eqref{eq:wts-alt} are also established.  The text \cite{silov} gives no references, but given the similarity of notation with \cite{Lowner} (compare \eqref{alphak2} with \eqref{alpha-k}), one could speculate that \v{S}ilov was influenced by L\"owner's work.

In a section \cite[Section 8.2]{paige_1971} of the PhD thesis of Paige entitled ``A Useful Theorem on Cofactors'', the identity \eqref{uij} is cited as ``a fascinating theorem ... that relates the elements of the eigenvectors of a symmetric to its eigenvalues and the eigenvalues of its principal submatrices'', with a version of the adjugate proof given.  In the notation of that thesis, one considers a $k \times k$ real symmetric tridiagonal matrix $C$ with distinct eigenvalues $\mu_1 > \dots > \mu_k$ with an orthonormal of eigenvectors $y_1,\dots,y_k$.  For any $0 \leq r < j \leq k$, let $C_{r,j}$ denote the $j-r \times j-r$ minor of $C$ defined by taking the rows and columns indexed between $r+1$ and $j$, and let $p_{i,j}(\mu) \coloneqq \mathrm{det}( \mu I_{j-r} - C_{r,j} )$ denote the associated trigonometric polynomial.  The identity
\begin{equation}\label{yri} y_{ri}^2 = p_{0,r-1}(\mu_i) p_{r,k}(\mu_i) / f(i)
\end{equation}
is then established for $i=1,\dots,n$, where $y_{ri}$ is the $i^{\mathrm{th}}$ component of $y_r$ and $f(i) \coloneqq \prod_{r=1; r \neq i}^k (\mu_i-\mu_r)$.  This is easily seen to be equivalent to \eqref{eq:wts} in the case of real symmetric tridiagonal matrices with distinct eigenvalues.  One can then use this to derive \eqref{eq:wts} for more general real symmetric matrices by a version of the Lanczos algorithm for tridiagonalizing an arbitrary real symmetric matrix, followed by the usual limiting argument to remove the hypothesis of distinct eigenvalues; we leave the details to the interested reader.  Returning to the case of tridiagonal matrices, Paige also notes that the method also gives the companion identity
\begin{equation}\label{yri-2}
 f(i) y_{ri} y_{si} = \delta_{r+1} \dots \delta_s p_{0,r-1}(\mu_i) p_{s,k}(\mu_i)
\end{equation}
for $1 \leq r < s \leq k$, where $\delta_2,\dots,\delta_k$ are the upper diagonal entries of the tridiagonal matrix $C$; this can be viewed as a special case of \eqref{off-diag}.  These identities were then used in \cite[Section 8]{paige_1971} as a tool to bound the behavior of errors in the symmetric Lanczos process.

Paige's identities \eqref{yri}, \eqref{yri-2} for tridiagonal matrices are reproduced in the textbook of Parlett \cite[Theorem 7.9.2]{Parlett}, with slightly different notation. Namely, one starts with an $n \times n$ real symmetric tridiagonal matrix $T$, decomposed spectrally as $S \Theta S^*$ where $S = (s_1,\dots,s_n)$ is orthogonal and $\Theta = \mathrm{diag}(\theta_1,\dots,\theta_n)$.  Then for $1 \leq \mu \leq \nu \leq n$ and $1 \leq j \leq n$, the $j^{\mathrm{th}}$ component $s_{\mu j}$ of $s_\mu$ is observed to obey the formula
$$ s_{\mu j}^2= \chi_{1:\mu-1}(\theta_j)  \chi_{\mu+1:n}(\theta_j) / \chi'_{1:n}(\theta_j)$$
when $\theta_j$ is a simple eigenvalue, where $\chi_{i:j}$ is the characteristic polynomial of the $j-i+1 \times j-i+1$ minor of $T$ formed by taking the rows and columns between $i$ and $j$. This identity is essentially equivalent to \eqref{yri}.  The identity \eqref{yri} is similarly reproduced in this notation; much as in \cite{paige_1971}, these identities are then used to analyze various iterative methods for computing eigenvectors.  The proof of the theorem is left as an exercise in \cite{Parlett}, with the adjugate method given as a hint.  Essentially the same result is also stated in the text of Golub and van Loan \cite[p. 432--433]{golub} (equation (8.4.12) on page 474 in the 2013 edition), proven using a version of the Cramer rule arguments in Section \ref{cramer-sec}; they cite as reference the earlier paper \cite[(3.6)]{golub1973}, which also uses essentially the same proof; see also \cite[(4.3.17)]{gladwell} (who cites \cite{BG}, who in turn cite \cite{golub1973}).
A similar result was stated without proof by Galais, Kneller, and Volpe \cite[Equations (6), (7)]{Galais:2011jh}.
They provided expressions for both $|v_{i,j}|^2$ and the off-diagonal eigenvectors as a function of cofactors in place of adjugate matrices.
Their work was in the context of neutrino oscillations.

The identities of Parlett and Golub-van Loan are cited in the thesis of Knyazev \cite[(2.2.27)]{knyazev}, again to analyze methods for computing eigenvalues and eigenvectors; the identities of Golub-van Loan and \v{S}ilov are similarly cited in the paper of Knyazev and Skorokhodov \cite{KS} for similar purposes. Parlett's result is also reproduced in the text of Xu \cite[(3.19)]{xu}.  In the survey \cite[(4.9)]{CG} of Chu and Golub on structured inverse eigenvalue problems, the eigenvector-eigenvalue identity is derived via the adjugate method from the results of \cite{Thompson:1968}, and used to solve the inverse eigenvalue problem for Jacobi matrices; the text \cite{Parlett} is also cited.

In the paper \cite[page 210]{DLNT} of Deift, Li, Nanda, and Tomei, the eigenvector-eigenvalue identity \eqref{eq:wts} is derived by the Cramer's rule method, and used to construct action-angle variables for the Toda flow.  The paper cites \cite{BG}, which also reproduces \eqref{eq:wts} as equation (1.5) of that paper, and in turn cites \cite{golub1973}.

In the paper of Mukherjee and Datta \cite{Mukherjee1989} the eigenvector-eigenvalue identity was rediscovered, in the context of computing eigenvectors of graphs that arise in chemistry.  If $G$ is a graph on $n$ vertices $v_1,\dots,v_n$, and $G - v_r$ is the graph on $n-1$ vertices formed by deleting a vertex $v_r, r=1,\dots,n$, then in \cite[(4)]{Mukherjee1989} the identity
\begin{equation}\label{mit}
 P(G - v_r; x_j) = P'(G; x_j) C^2_{rj}
\end{equation}
is established for $j,r=1,\dots,n$, where $P(G;x)$ denotes the characteristic polynomial of the adjacency matrix of $G$ evaluated at $x$, and $C_{rj}$ is the coefficient at the $r^{\mathrm{th}}$ vertex of the eigenvector corresponding to the $j^{\mathrm{th}}$ eigenvalue, and one assumes that all the eigenvalues of $G$ are distinct.  This is equivalent to \eqref{eq:wts-alt} in the case that $A$ is an adjacency matrix of a graph.  The identity is proven using the perturbative method in Section \ref{perturbative-sec}, and appears to have been discovered independently.  A similar identity was also noted in the earlier work of Li and Feng \cite{LF}, at least in the case $j=1$ of the largest eigenvalue.
In a later paper of Hagos \cite{HAGOS2002103}, it is noted that the identity \eqref{mit} ``is probably not as well known as it should be'', and also carefully generalizes \eqref{mit} to an identity (essentially the same as \eqref{lam}) that holds when some of the eigenvalues are repeated.  An alternate proof of \eqref{mit} was given in the paper of Cvetkovic, Rowlinson, and Simic \cite[Theorem 3.1]{CVETKOVIC2007146}, essentially using the Cramer rule type methods in Section \ref{cramer-sec}. The identity \eqref{lkmj} is also essentially noted at several other locations in the graph theory literature, such as \cite[Chapter 4]{godsil-book}, \cite[Lemma 2.1]{GK}, \cite[Lemma 7.1, Corollary 7.2]{godsil}, \cite[(2)]{GGKL} in relation to the generating functions for walks on a graph, though in those references no direct link to the eigenvector-eigenvalue identity in the form \eqref{eq:wts} is asserted.

In \cite[Section 2]{NTU} the identity \eqref{eq:wts} is derived for normal matrices by the Cramer rule method, citing \cite{thompson-iv}, \cite{deutsch} as the source for the key identity \eqref{lkmj}; the papers \cite{Thompson:1966}, \cite{Thompson:1968} also appear in the bibliography but were not directly cited in this section.  An extension to the case of eigenvalue multiplicity, essentially corresponding to \eqref{lam}, is also given.  This identity is then used to give a complete description of the relations between the eigenvalues of $A$ and of a given minor $M_j$ when $A$ is assumed to be normal.  In \cite{BFdP} a generalization of these results was given to the case of $J$-normal matrices for some diagonal sign matrix $J$; this corresponds to a special case of \eqref{adji-3} in the case where each left eigenvector $w_i$ is the complex conjugate of $Jv_i$.
 
The paper of Baryshnikov \cite{Bary} marks the first appearance of this identity in random matrix theory.  Let $H$ be a Hermitian form on $\C^M$ with eigenvalues $\lambda_1 \geq \dots \geq \lambda_M$, and let $L$ be a hyperplane of $\C^M$ orthogonal to some unit vector $l$.  Let $l_i$ be the component of $l$ with respect to an eigenvector $v_i$ associated to $\lambda_i$, set $w_i \coloneqq |l_i|^2$, and let $\mu_1 \geq \dots \geq \mu_{M-1}$ be the eigenvalues of the Hermitian form arising from restricting $H$ to $L$.  Then after \cite[(4.5.2)]{Bary} (and correcting some typos) the identity
$$ w_i = \frac{\prod_{1 \leq j \leq M-1}(\lambda_i-\mu_j)}{\prod_{1 \leq j \leq M; j \neq i} (\lambda_i-\lambda_j)}$$
is established, by an argument based on Cramer's rule and the Cauchy determinant formula \eqref{cauchy-det}, similar to the arguments at the end of Section \ref{cramer-sec}, and appears to have been discovered independently.  If one specializes to the case when $l$ is a standard basis vector $e_j$ then $l_i$ is also the $e_j$ component of $v_i$, and we recover \eqref{eq:wts} after a brief calculation.  This identity was employed in \cite{Bary} to study the situation in which the hyperplane normal $l$ was chosen uniformly at random on the unit sphere.  This formula was rederived (using a version of the Cramer rule method in Section \ref{cramer-sec}) in the May 2019 paper of Forrester and Zhang \cite[(2.7)]{2019arXiv190505314F}, who recover some of the other results in \cite{Bary} as well, and study the spectrum of the sum of a Hermitian matrix and a random rank one matrix.

In the paper \cite[Lemma 2.7]{DE} of Dumitriu and Edelman, the identity \eqref{yri} of Paige (as reproduced in \cite[Theorem 7.9.2]{Parlett}) is used to give a clean expression for the Vandermonde determinant of the eigenvalues of a tridiagonal matrix, which is used in turn to construct tridiagonal models for the widely studied \emph{$\beta$-ensembles} in random matrix theory.

In the unpublished preprint \cite{2014arXiv1401.4580V} of Van Mieghem, the identity \eqref{eq:wts-alt} is prominently displayed as the main result, though in the notation of that preprint it is expressed instead as
$$ (x_k)_j^2 = -\frac{1}{c'_A(\lambda_k)} \mathrm{det}( A_{\backslash \{j\}} - \lambda_k I_n )$$
for any $j,k=1,\dots,n$, where $A$ is a real symmetric matrix with distinct eigenvalues $\lambda_1,\dots,\lambda_n$ and unit eigenvectors $x_1,\dots,x_n$, $A_{\backslash \{j\}}$ is the minor formed by removing the $j^{\mathrm{th}}$ row and column from $A$, and $c'_A$ is the derivative of the (sign-reversed) characteristic polynomial $c_A(\lambda) = \mathrm{det}(A - \lambda I_n) = (-1)^n p_A(\lambda)$.  Two proofs of this identity are given, one being essentially the Cramer's rule proof from Section \ref{cramer-sec} and attributed to the previous reference \cite{CVETKOVIC2007146}; the other proof is based on Cramer's rule and the Desnanot-Jacobi identity (Dodgson condensation); this identity is used to quantify the effect of removing a node from a graph on the spectral properties of that graph.  The related identity \eqref{wts-3} from \cite{tao2011} is also noted in this preprint.  Some alternate formulae from \cite{vmbook} for quantities such as $(x_k)_j^2$ in terms of walks of graphs are also noted, with the earlier texts \cite{godsil-book}, \cite{golub} also cited.

The identity \eqref{eq:wts} was independently rediscovered and then generalized by Kausel \cite{Kausel2018}, as a technique to extract information about components of a generalized eigenmode without having to compute the entire eigenmode.  Here the generalized eigenvalue problem
$$ \mathbf{K} \psi_j = \lambda_j \mathbf{M} \psi_j$$
for $j=1,\dots,N$ is considered, where $\mathbf{K}$ is a positive semi-definite $N \times N$ real symmetric matrix, $\mathbf{M}$ is a positive definite $N \times N$ real symmetric matrix, and the matrix $\mathbf{\Psi} = (\psi_1 \dots \psi_N)$ of eigenfunctions is normalized so that $\mathbf{\Psi}^T \mathbf{M} \mathbf{\Psi} = \mathbf{I}$.  For any $1 \leq \alpha \leq N$, one also solves the constrained system
$$ \mathbf{K}_\alpha \psi_{j}^{(\alpha)} = \lambda_{\alpha j} \mathbf{M}_\alpha \psi_j^{(\alpha)}$$
where $\mathbf{K}_\alpha, \mathbf{M}_\alpha$ are the $N-1 \times N-1$ minors of $\mathbf{K}, \mathbf{M}$ respectively formed by removing the $\alpha^{\mathrm{th}}$ row and column.  Then in \cite[(18)]{Kausel2018} the Cramer rule method is used to establish the identity
$$ \psi_{\alpha j} = \pm \sqrt{\frac{|\mathbf{M}_\alpha|}{|\mathbf{M}|}} \sqrt{\frac{ \prod_{k=1}^{N-1} (\lambda_j - \lambda_{\alpha k})}{\prod_{k=1; k \neq j}^{N-1} (\lambda_j-\lambda_k)}}$$
for the $\alpha$ component $\psi_{\alpha j}$ of $\psi_j$, where $|\mathbf{M}|$ is the notation in \cite{Kausel2018} for the determinant of $\mathbf{M}$.  Specializing to the case when $\mathbf{M}$ is the identity matrix, we recover \eqref{eq:wts}.

The eigenvector-eigenvalue identity was discovered by three of us \cite{Denton:2019ovn} in July 2019, initially in the case of $3 \times 3$ matrices, in the context of trying to find a simple and numerically stable formula for the eigenvectors of the neutrino oscillation Hamiltonian, which form a separate matrix known as the PMNS lepton mixing matrix.  	This identity was established in the $3 \times 3$ case by direct calculation.  Despite being aware of the related identity \eqref{wts-3}, the four of us were unable to locate this identity in past literature and wrote a preprint \cite{DPTZ} in August 2019 highlighting this identity and providing two proofs (the adjugate proof from Section \ref{adjugate-sec}, and the Cauchy-Binet proof from Section \ref{cauchy-binet-sec}).  The release of this preprint generated some online discussion\footnote{\tt terrytao.wordpress.com/2019/08/13, www.reddit.com/r/math/comments/cq3en0}, and we were notified by Jiyuan Zhang (private communication) of the prior appearance of the identity earlier in the year in \cite{2019arXiv190505314F}.  However, the numerous other places in the literature in which some form of this identity appeared did not become revealed until a popular science article \cite{wolchover-2019} by Wolchover was written in November 2019.  This article spread awareness of the eigenvector-eigenvalue identity to a vastly larger audience, and generated a large number of reports of previous occurrences of the identity, as well as other interesting related observations, which we have attempted to incorporate into this survey.

\section{Further discussion}\label{discuss-sec}

The eigenvector-eigenvalue identity \eqref{eq:wts} only yields information about the magnitude $|v_{i,j}|$ of the components of a given eigenvector $v_i$, but does not directly reveal the phase of these components.  On one hand, this is to be expected, since (as already noted in the consistency check (vii) in the introduction) one has the freedom to multiply $v_i$ by a phase; for instance, even if one restricts attention to real symmetric matrices $A$ and requires the eigenvectors to be real $v_i$, one has the freedom to replace $v_i$ by its negation $-v_i$, so the sign of each component $v_{i,j}$ is ambiguous.  However, \emph{relative} phases, such as the phase of $v_{i,j} \overline{v_{i,j'}}$ are not subject to this ambiguity.  There are several ways to try to recover these relative phases.  One way is to employ the off-diagonal analogue \eqref{off-diag} of \eqref{eq:wts}, although the determinants in that formula may be difficult to compute in general.  For small matrices, it was suggested in \cite{Mukherjee1989} that the signs of the eigenvectors could often be recovered by direct inspection of the components of the eigenvector equation $A v_i = \lambda_i(A) v_i$.  In the application in \cite{Denton:2019ovn}, the additional phase could be recovered by a further neutrino specific identity \cite{Toshev:1991ku}.  For more general matrices, one way to retrieve such phase information is to apply \eqref{eq:wts} in multiple bases.  For instance, suppose $A$ was real symmetric and the $v_{i,j}$ were all real.  If one were to apply the eigenvector-eigenvalue identity after changing to a basis that involved the unit vector $\frac{1}{\sqrt{2}} (e_j + e_{j'})$, then one could use the identity to evaluate the magnitude of $\frac{1}{\sqrt{2}} ( v_{i,j} + v_{i,j'} )$.  Two further applications of the identity in the original basis would give the magnitude of $v_{i,j}, v_{i,j'}$, and this is sufficient information to determine the relative sign of $v_{i,j}$ and $v_{i,j'}$.  We also remark that for real symmetric matrices that are acyclic (such as weighted adjacency matrices of graphs that do not contain loops), one can write down explicit formulae for the coefficients of eigenvectors (and not just their magnitudes) in terms of characteristic polynomials of minors; see \cite{BK}.  We do not know of any direct connection between such formulae and the eigenvector-eigenvalue identity \eqref{eq:wts}.

For large unstructured matrices, it does not seem at present that the identity \eqref{eq:wts} provides a competitive algorithm to compute eigenvectors.  Indeed, to use this identity to compute all the eigenvector component magnitudes $|v_{i,j}|$, one would need to compute all $n-1$ eigenvalues of each of the $n$ minors $M_1,\dots,M_n$, which would be a computationally intensive task in general; and furthermore, an additional method would then be needed to also calculate the signs or phases of these components.  However, if the matrix is of a special form (such as a tridiagonal form), then the identity could be of more practical use, as witnessed by the uses of this identity (together with variants such as \eqref{yri-2}) in the literature to control the rate of convergence for various algorithms to compute eigenvalues and eigenvectors of tridiagonal matrices.  Also, as noted recently in \cite{Kausel2018}, if one has an application that requires only the component magnitudes $|v_{1,j}|,\dots,|v_{n,j}|$ at a single location $j$, then one only needs to compute the characteristic polynomial of a single minor $M_j$ of $A$ at a single value $\lambda_i(A)$, and this may be more computationally feasible.

\section{Sociology of science issues}

As one sees from Section \ref{history-sec} and Figure \ref{fig:graph}, there was some partial dissemination of the eigenvector-eigenvalue identity amongst some mathematical communities, to the point where it was regarded as ``folklore'' by several of these communities.  However, this process was unable to raise broader awareness of this identity, resulting in the remarkable phenomenon of multiple trees of references sprouting from independent roots, and only loosely interacting with each other.  For instance, as discussed in the previous section, for two months after the release of our own preprint \cite{DPTZ}, we only received a single report of another reference \cite{2019arXiv190505314F} containing a form of the identity, despite some substantial online discussion and the dozens of extant papers on the identity.  It was only in response to the popular science article \cite{wolchover-2019} that awareness of the identity finally ``went viral'', leading to what was effectively an \emph{ad hoc} crowdsourced effort to gather all the prior references to the identity in the literature.    
While we do not know for certain why this particular identity was not sufficiently well known prior to these recent events, we can propose the following possible explanations:

\begin{enumerate}
\item \emph{The identity was mostly used as an auxiliary tool for other purposes}.  In almost all of the references discussed here, the eigenvector-eigenvalue identity was established only in order to calculate or bound some other quantity; it was rarely formalized as a theorem or even as a lemma.  In particular, with a few notable exceptions such as the preprint \cite{2014arXiv1401.4580V}, this identity would not be mentioned in the title, abstract, or even the introduction.  In a few cases, the identity was reproven by authors who did not seem to be fully aware that it was already established in one of the references in their own bibliography.
\item \emph{The identity does not have a standard name, form, or notation, and does not involve uncommon keywords}.  As one can see from Section \ref{history-sec}, the identity comes in many variants and can be rearranged in a large number of ways; furthermore, the notation used for the various mathematical objects appearing in the identity vary greatly depending on the intended application, or on the authors involved.  Also, none of the previous references attempted to give the identity a formal name, and the keywords used to describe the identity (such as ``eigenvector'' or ``eigenvalue'') are in extremely common use in mathematics.  As such, there are no obvious ways to use modern search engines to locate other instances of this identity, other than by manually exploring the citation graph around known references to that identity.  Perhaps a ``fingerprint database'' for identities \cite{BT} would be needed before such automated searches could become possible.
\item \emph{The field of linear algebra is too mature, and its domain of applicability is too broad}.  The vast majority of consumers of linear algebra are not domain experts in linear algebra itself, but instead use it as a tool for a very diverse array of other applications.  As such, the diffusion of linear algebra knowledge is not guided primarily by a central core of living experts in the field, but instead relies on more mature sources of authority such as textbooks and lectures.  Unfortunately, only a small handful of linear algebra textbooks mention the eigenvector-eigenvalue identity, thus preventing wider dissemination of this identity.
\end{enumerate}

Online discussion forums for mathematics were only partially successful in disseminating this identity.  For instance, the 2012 MathOverflow question\footnote{{\tt mathoverflow.net/questions/96190}} ``Cramer's rule for eigenvectors'', which inquired as to the existence of an eigenvector identity such as \eqref{eq:wts}, received nearly ten thousand views, but only revealed the related identity in Lemma \ref{vn}.  Nevertheless, this post was instrumental in bringing us four authors together to produce the preprint \cite{DPTZ}, via a comment\footnote{{\tt www.reddit.com/r/math/comments/ci665j/linear\_algebra\_question\_from\_a\_physicist/ev22xgp/}} on a Reddit post by one of us.

It is not fully clear to us how best to attribute authorship for the eigenvector-eigenvalue identity \eqref{eq:wts}.  A variant of the identity was observed by Jacobi \cite{jacobi}, but not widely propagated.  An identity that implies \eqref{eq:wts} was later given by L\"owner \cite{Lowner}, but the implication is not immediate, and this reference had only a modest impact on the subsequent literature.  The paper of Thompson \cite{Thompson:1966} is the first  place we know of in which the identity explicitly appears, and it was propagated through citations into several further papers in the literature; but this did not prevent the identity from then being independently rediscovered several further times, such as in the text \cite{golub} (with the latter restricting attention to the case of tridiagonal matrices).  Furthermore, we are not able to guarantee that there is not an even earlier place in the literature where some form of this identity has appeared.  We propose the name ``eigenvector-eigenvalue identity'' for \eqref{eq:wts} on the grounds that it is descriptive, and hopefully is a term that can be detected through search engines by researchers looking for identities of this form.

%Although, in this survey we have included approximately 50 references that mention some variant of the eigenvector-eigenvalue identity, in most cases the identity does not explicitly appear in a form such as \eqref{eq:wts} that specifically links eigenvector component magnitudes of an arbitrary Hermitian matrix to eigenvalues of the matrix and its minors; exceptions include the paper \cite{Thompson:1966} and the text \cite{golub} (with the latter restricting attention to the case of tridiagonal matrices), as well as several other works citing these references (either directly or transitively).  To convert the other forms of the identity appearing in the literature to a form similar to \eqref{eq:wts} requires a small but non-zero amount of additional work (such as a change of basis, passing to a limit, or expressing a characteristic polynomial or determinant in terms of eigenvalues).  This may well be an additional factor that has prevented this identity from being more widely known until recently.

Although, in this survey we have included approximately 50 references that mention some variant of the eigenvector-eigenvalue identity, in most cases the identity does not explicitly appear in a form such as  \eqref{eq:wts} that specifically links eigenvector component magnitudes of an arbitrary Hermitian matrix to eigenvalues of the matrix and its minors; exceptions include the papers \cite{Thompson:1966} , \cite{Thompson:1968}, \cite{NTU}, and (in the special case of tridiagonal matrices) \cite{golub}, \cite{xu}.  To convert the other forms of the identity appearing in the literature to a form similar to  \eqref{eq:wts} requires a small but non-zero amount of additional work (such as a change of basis, passing to a limit, or expressing a characteristic polynomial or determinant in terms of eigenvalues).  This may well be an additional factor that has prevented this identity from being more widely known until recently.

\section{Acknowledgments}\label{ack}

We thank Asghar Bahmani, Carlo Beenakker, Adam Denchfield, Percy Deift, Laurent Demanet, Alan Edelman, Chris Godsil, Aram Harrow, James Kneller, Andrew Knysaev, Manjari Narayan, Michael Nielsen, Ma{\l}gorzata Stawiska, Karl Svozil, Gang Tian, Carlos Tomei, Piet Van Mieghem, Brad Willms, Fu Zhang, Jiyuan Zhang, and Zhenzhong Zhang for pointing out a number of references where some variant of the eigenvector-eigenvalue identity has appeared in the literature, or suggesting various extensions and alternate proofs of the identity.  We also thank Darij Grinberg, Andrew Krause, Kristoffer Varholm, Jochen Voss and Jim Van Zandt for some corrections to previous versions of this manuscript.

\bibliographystyle{alpha_mod}
\bibliography{EigenSurvey}

\end{document}